\documentclass[12pt,reqno]{amsart}
\usepackage{enumerate, latexsym, amsmath, amsfonts, amssymb, amsthm, graphicx, color}
\usepackage{hyperref}
\usepackage{amsmath}
\def\pmod #1{\ ({\rm{mod}}\ #1)}

\def\bg{\bigg}
\def\({\bg(}
\def\){\bg)}

\theoremstyle{plain}
\newtheorem{theorem}{Theorem}

\newtheorem{lemma}{Lemma}

\theoremstyle{definition}

\newtheorem*{acknowledgment}{Acknowledgment}
\theoremstyle{remark}
\newtheorem{remark}{Remark}

\makeatletter
\@namedef{subjclassname@2020}{%
	\textup{2020} Mathematics Subject Classification}
\makeatother
\vspace{4mm}

\begin{document}

	\title[A note on additive complements of squares]{A note on additive complements of squares}
	
	\author[Y. Ding, Y.--C. Sun, L.--Y. Wang and Y. Xia]{Yuchen Ding, Yu--Chen Sun, Li--Yuan Wang and Yutong Xia}
	
	\address {(Yuchen Ding) School of Mathematical Science,  Yangzhou University, Yangzhou 225002,  People's Republic of China}
	\email{\tt ycding@yzu.edu.cn}
	
	\address {(Yu--Chen Sun) Department of Mathematics and Statistics, University of Turku, Turku 20014 , Finland}
	\email{\tt yuchensun93@163.com}
	
	\address {(Li-Yuan Wang) School of Physical and Mathematical Sciences, Nanjing Tech University, Nanjing 211816, People's Republic of China}
	\email{\tt wly@smail.nju.edu.cn}
	
	\address {(Yutong Xia) School of Mathematical Science,  Yangzhou University, Yangzhou 225002,  People's Republic of China}
	\email{\tt 1220045260@qq.com}
	
	\begin{abstract}
		Let $\mathcal{S}=\{1^2,2^2,3^2,...\}$ be the set of squares and $\mathcal{W}=\{w_n\}_{n=1}^{\infty} \subset \mathbb{N}$ be an additive complement of $\mathcal{S}$ so that $\mathcal{S} + \mathcal{W} \supset \{n \in \mathbb{N}: n \geq N_0\}$ for some $N_0$. Let $\mathcal{R}_{\mathcal{S},\mathcal{W}}(n) = \#\{(s,w):n=s+w, s\in \mathcal{S}, w\in \mathcal{W}\} $.
		
		In 2017, Chen-Fang \cite{C-F} studied the lower bound of $\sum_{n=1}^NR_{\mathcal{S},\mathcal{W}}(n)$.
		In this note, we improve Cheng-Fang's result and get that
		$$\sum_{n=1}^NR_{\mathcal{S},\mathcal{W}}(n)-N\gg N^{1/2}.$$
		As an application, we make some progress on a problem of Ben Green problem by showing that
		$$\limsup_{n\rightarrow\infty}\frac{\frac{\pi^2}{16}n^2-w_n}{n}\ge \frac{\pi}{4}+\frac{0.193\pi^2}{8}.$$
	\end{abstract}
	
	
	\maketitle
	
	\section{Introduction}
	\setcounter{lemma}{0}
	\setcounter{theorem}{0}
	\setcounter{corollary}{0}
	\setcounter{remark}{0}
	\setcounter{equation}{0}
	\setcounter{conjecture}{0}
	Let $\mathbb{N}$ be the set of all non-negative integers.
	We say that $A,B \subset \mathbb{N}$ are {\it additive complements} if $A+B \supset \{n\in\mathbb{N}: n \geq N_0\}$ for some fixed $N_0$. If $A,B$ are additive complements, we say that $A$ is an additive complement of $B$.
	Let $\mathcal{S}=\{1^2,2^2,3^2,...\}$ be the set of squares and $\mathcal{W} \subset \mathbb{N}$ be an additive complement of $\mathcal{S}$.
	Let $N$ be a large integer and $\mathcal{W}(N)$ be the number of elements of $\mathcal{W}$ not exceeding $N$. Clearly, we have
	$\mathcal{W}(N)\sqrt{N}\ge N-N_0$ for some given integer $N_0$, from which we deduce that
	$$\liminf_{N\rightarrow\infty}\mathcal{W}(N)/\sqrt{N}\ge 1.$$
	It is of great interest to ask whether $\liminf_{N\rightarrow\infty}\mathcal{W}(N)/\sqrt{N}$ is strictly larger than $1$.
	
	In 1956, Erd\H os \cite{Erd} posed this problem, which was later settled affirmatively by Moser \cite{Mos} with an accurate lower bound $1.06$. The number was then improved in a few articles \cite{Abb,Bal,B-S,D-H,Ram,Ram2}. Up to now the best result is
	\begin{align}\label{eq1-1}
		\liminf_{N\rightarrow\infty}\mathcal{W}(N)/\sqrt{N}\ge 4/\pi,
	\end{align}
	given by Cilleruelo \cite{Cil}, Habsieger \cite{Hab}, Balasubramanian and Ramana \cite{B-R}.
	
	Note that $\mathcal{W}(N)/\sqrt{N} >1$ implies that there are some integers $n \in \mathbb{N}$ which do not have a unique  representation. Thus it is worthwhile to study the number of representations for writing $n$ as a sum $s+w$ with $s \in \mathcal{S}$ and $w \in \mathcal{W}$.
	
	For any $A,B \subset \mathbb{N}$, let
	$$R_{A,B}(n)=\#\{(a,b):n=a+b,a\in A,b\in B\}.$$
	Obviously, if $A,B$ are complements, then
	$$\liminf_{N\rightarrow\infty}\frac{1}{N}\sum_{n\le N}R_{A,B}(n)\ge 1.$$
	Ben Green had a nice observation that for $B=\{b_1,b_2,...\}$ with $b_n=\frac{\pi^2}{16}n^2+o(n^2)$,
	$$\lim_{N\rightarrow \infty}\frac{1}{N}\sum_{n=1}^NR_{\mathcal{S},B}(n)=1.$$
	Later, he asked Fang (private communication, see \cite{C-F}) whether there is an additive complement of $\mathcal{S}$ has the same property as $B$. Namely, can we find an additive complement $\mathcal{W}=\{w_n\}_{n=1}^{\infty}$ of $\mathcal{S}$ satisfying the asymptotic condition $w_n=\frac{\pi^2}{16}n^2+o(n^2)$?
	
	Motivated by Ben Green's problem, Chen and Fang \cite[Theorem 1.1]{C-F} proved the following result.
	
	Let $N$ be a sufficiently large integer. For any additive complement $\mathcal{W}$ of $\mathcal{S}$, we have
	\begin{align}\label{eq1-2}
		\sum_{n=1}^NR_{\mathcal{S},\mathcal{W}}(n)-N\ge \frac{1}{2\log 4}\mathcal{W}(2\sqrt{N})\log \mathcal{W}(2\sqrt{N})
	\end{align}
	From (\ref{eq1-1}), one can immediately obtain the following corollary:
	If $\mathcal{W}$ is any complement of $\mathcal{S}$, then
	\begin{equation}\label{CFmain}
		\sum_{n=1}^NR_{\mathcal{S},\mathcal{W}}(n)-N\gg N^{1/4}\log N.
	\end{equation}
	
	Chen and Fang \cite[Remark 1]{C-F} also gave a counterexample to show that (\ref{eq1-2}) does not always hold, if $\mathcal{S}$ and $\mathcal{W}$ are substituted with general additive complements $A,B$.
	
	As an application of (\ref{eq1-2}), Chen and Fang considered Ben Green's problem and showed that
	$$\limsup_{n\rightarrow\infty}\frac{\frac{\pi^2}{16}n^2-w_n}{n^{1/2}\log n}\ge\sqrt{\frac{2}{\pi}}\frac{1}{\log 4}$$
	for any additive complement $\mathcal{W}=\{w_n\}_{n=1}^{\infty}$ of $\mathcal{S}$. They further conjectured that
	$$\limsup_{n\rightarrow\infty}\frac{\frac{\pi^2}{16}n^2-w_n}{n^{1/2}\log n}=\infty,$$
	which was later confirmed by the first author \cite{Ding} with the following stronger form
	\begin{equation}\label{ding_lim}
		\limsup_{n\rightarrow\infty}\frac{\frac{\pi^2}{16}n^2-w_n}{n}\ge \frac{\pi}{4}.
	\end{equation}
	The result (\ref{ding_lim}) only use the trivial bound
	$$\sum_{n=1}^NR_{\mathcal{S},\mathcal{W}}(n)-N\ge n_0$$
	for some negative integer $n_0$, but if one can replaced $\ge n_0$ by $\gg N^{1/2}$, then one can improve (\ref{ding_lim}), see the proof of Theorem \ref{theorem2}.

	In this note, we first give an improvement of  (\ref{CFmain}) and obtain the following theorem.
	\begin{theorem}\label{theorem1}
		Let $\mathcal{W}$ be an additive complement of $\mathcal{S}$. For sufficiently large  integers $N$ (depending on $\mathcal{W}$) we have
		$$\sum_{n=1}^NR_{\mathcal{S},\mathcal{W}}(n)-N\ge0.193 N^{1/2}.$$
	\end{theorem}
	\begin{remark}
		Cilleruelo \cite{Cil} conjectured that
		$$\sum_{n=1}^NR_{\mathcal{S},\mathcal{W}}(n)-N\ge N+o(N).$$  Our new bound in Theorem \ref{theorem1} makes some progress to Cilleruelo's conjecture but the resolution of this conjecture is apparently too much to hope for at present.
	\end{remark}
	By applying Theorem \ref{theorem1} and the arguments in \cite{Ding}, we can improve on (\ref{ding_lim}).
	\begin{theorem}\label{theorem2} For any additive complement $\mathcal{W}=\{w_n\}_{n=1}^{\infty}$ of $S$, we have
		\begin{align}\label{green2}
			\limsup_{n\rightarrow\infty}\frac{\frac{\pi^2}{16}n^2-w_n}{n}\ge \frac{\pi}{4}+\frac{0.193\pi^2}{8}.
		\end{align}
	\end{theorem}
	\begin{remark}
		The value $\frac{\pi}{4}+\frac{0.193\pi^2}{8}=1.0235\dots$ should be compared with $\frac{\pi}{4}=0.7853\cdots$ in \cite{Ding}. It would be interesting to show that the left hand side of (\ref{green2}) is $\infty$.
	\end{remark}

	The proof of Theorem \ref{theorem1} is based on the structure of the proof in Chen--Fang \cite{C-F}. In Chen--Fang's proof, for $d_1,d_2 \in \mathcal{W}$, they considered the equation
	\begin{equation}\label{xy_eq}
		x^2+d_1 = y^2+ d_2 < N
	\end{equation}
	and required that this equation has $\gg \log N$ solutions (see \cite[Lemma 2.1]{C-F}). Thus the $\log$-factor, in (\ref{CFmain}), comes from the number of solutions. However, in Chen-Fang's arguments, the above equation has some solutions $(x,y)$ such that $|x-y|$ is very small, namely $O(1)$, and thus $x+y$ is very large, namely $O(|d_2 -d_1|)$. Since $\max\{x^2,y^2\} < N$, we must have $|d_2-d_1| \ll N^{1/2}$. For this reason, in their arguments, they restricted $d_1, d_2 \in \mathcal{W} \cap [cN^{1/2}]$ for some $c>0$. Here $[N]$ denotes the set $\{n\in \mathbb{N}: n\leq N\}$.
	
	Assume that $x>y>0$ in (\ref{xy_eq}). Our new idea is that we can always find at least one solution $(x,y)$ such that $N^{1/2} \ll x < N^{1/2}$, see (\ref{solution}). When we do this restriction, we will see that we can enlarge Cheng-Fang's $\mathcal{W} \cap [cN^{1/2}]$ to $\mathcal{W} \cap [\epsilon_0 N]$ with some small constant $\epsilon_0$. Hence we can utilize the lower bound of $|\mathcal{W} \cap [\epsilon_0 N]|$ instead of using the lower bound of  $|\mathcal{W} \cap [cN^{1/2}]|$ to obtain Theorem \ref{theorem1}. But the payoff is that since we just pick solutions $(x,y)$ such that $|x-y|$ is large, we can only show that the number of available solutions in our case is $\gg 1$ instead of $\gg \log N$.

	\begin{acknowledgment}
		The authors would like to thank Kaisa Matom\"{a}ki for her helpful comments.
		
		The first author was supported by National Natural Science Foundation of China (Grant No. 12201544), Natural Science Foundation of Jiangsu Province of China (Grant No. BK20210784) and China Postdoctoral Science Foundation (Grant No. 2022M710121). He was also supported by foundation numbers JSSCBS20211023 and YZLYJF2020PHD051. The second author was supported by UTUGS funding and was working in the Academy of Finland project No. $333707$. The third author was supported by the National Natural Science Foundation of China (Grant No. 12201291) and the Natural Science Foundation of the Higher Education Institutions of Jiangsu Province (21KJB110001).
	\end{acknowledgment}
	\maketitle
	
	\section{Proof of Theorem \ref{theorem1}}
	\setcounter{lemma}{0}
	\setcounter{theorem}{0}
	\setcounter{corollary}{0}
	\setcounter{remark}{0}
	\setcounter{equation}{0}
	\setcounter{conjecture}{0}
	The proof of our theorem relies on the following lemma, which is a refinement of the original argument of Chen and Fang \cite[Theorem 2.1]{C-F}.
	\begin{lemma}\label{lem1} Let $\delta$ and $\delta_0$ be two positive numbers satisfying
		\begin{align*}
			\begin{cases}
				\delta^2+\delta_0\le 1, &\\
				\frac{1}{16}\delta_0^2/\delta^2+\delta_0<1
			\end{cases}
		\end{align*}
		and $K = \lfloor \delta N^{1/2} \rfloor$ a positive integer and $\mathcal{D}\subseteq\mathbb{N}$ be such that $4K|d-d'$ for any $d,d' \in \mathcal{D}$.
		Then for all sufficiently large integers $N$, we have
		$$\sum_{\substack{n \le N\\ R_{\mathcal{S},\mathcal{D}}(n)\ge 1}}\left(R_{\mathcal{S},\mathcal{D}}(n)-1\right)\ge \mathcal{D}\left(\delta_0 N \right)-2.$$
	\end{lemma}
	\begin{proof}
		The lemma is trivial if $ \mathcal{D}(\delta_0 N) \le 2$. So we only need to consider the case $ \mathcal{D}(\delta_0 N) > 2$. In this case, we assume $  \mathcal{D}(\delta_0 N)=\ell$ and
		$$\mathcal{D}\cap\left[0,\delta_0 N \right]=\{d_1<d_2<\cdot\cdot\cdot<d_\ell\}.$$
		For any $1<s\le \ell$, we have $4K \mid d_s-d_1$, it follows that for at least $\ell-2$ positive integers $n_s$ we have $d_s-d_1=4K n_s$ with $ n_s \neq K$. It can be observed that one of the solutions of the equation
		$$x_s^2-y_s^2=d_s-d_1=4K n_s$$
		has the form
		\begin{equation}\label{solution}
			\begin{cases}
				x_s=K+n_s, \\
				y_s=|K-n_s|,
			\end{cases}
		\end{equation}
		from which we deduce that
		$$n_s=\frac{d_s-d_1}{4K}< \frac{\delta_0N}{4(\delta N^{1/2}-1)}<\left(\frac{\delta_0}{4\delta}+\varepsilon\right)N^{1/2},$$
		provided that $N$ is sufficiently large, where $\varepsilon>0$ is an arbitrarily small number which may not be the same throughout our proof.
		Moreover, the solution given by (\ref{solution}) satisfies
		\begin{align}\label{eq1}
			x_s^2+d_1=y_s^2+d_s < N,
		\end{align}
		because for $\delta$ and $\delta_0$ satisfying the constraints of our lemma, we have
		\begin{align*}
			y_s^2+d_s&\le |K-n_s|^2 + \delta_0 N \\
			&\le \max\{K^2,n_s^2\}+ \delta_0 N\\
			&< \max\left\{\delta^2 N, \left(\frac{\delta_0}{4\delta}+\varepsilon\right)^2N \right\} + \delta_0 N\\
			& < N
		\end{align*}
		once $\varepsilon$ is small enough in terms of $\delta$ and $\delta_0$. If an integer $n\le N$ can be written as the sum of $d_1$ and a square, then
		\begin{align*}
			R_{\mathcal{S},\mathcal{D}}(n)-1\ge\sum_{\substack{d_s>d_1\\ n-d_s\in \mathcal{S}}}1,
		\end{align*}
		from which we deduce that
		\begin{equation}\label{eq3}
			\sum_{\substack{n \le N\\ R_{\mathcal{S},\mathcal{D}}(n)\ge 1}}\!\!\!\left(R_{\mathcal{S},\mathcal{D}}(n)-1\right) \ge\!\!\! \sum_{\substack{n \le N\\ n-d_1\in \mathcal{S}}}\!\!\!\left(R_{\mathcal{S},\mathcal{D}}(n)-1\right)
			\ge \!\!\!\sum_{\substack{n \le N\\ n-d_1\in \mathcal{S}}}\sum_{\substack{d_s>d_1\\ n-d_s\in \mathcal{S}}}1 = \sum_{\substack{1<s\le \ell}}\sum_{\substack{n \le N\\n-d_s\in \mathcal{S}\\ n-d_1\in \mathcal{S}}}1.
		\end{equation}
		For $1<s\le \ell$ with $n_s \neq K$, by (\ref{solution}) and (\ref{eq1}) we have $n=x_{s}^2+d_1$ which satisfies $n<N$, $n-d_1\in \mathcal{S}$ and $n-d_s\in \mathcal{S}$. 
		It follows by (\ref{eq3}) that,
		$$\sum_{\substack{n \le N\\ R_{\mathcal{S},\mathcal{D}}(n)\ge 1}}\left(R_{\mathcal{S},\mathcal{D}}(n)-1\right)\ge\sum_{\substack{1<s\le \ell \\ n_s \neq K}}\sum_{\substack{n \le N\\n-d_s\in \mathcal{S}\\ n-d_1\in \mathcal{S}}}1 \ge\sum_{\substack{1<s\le \ell \\ n_s \neq K}}1\ge\ell-2=\mathcal{D}\left(\delta_0 N \right)-2.$$
	\end{proof}
	Now, we turn to prove Theorem \ref{theorem1}.
	\begin{proof}[Proof of Theorem \ref{theorem1}]
		Let $N$ be a sufficiently large integer and $\mathcal{W}$ an additive complement of $\mathcal{S}$.
		Let $K = \lfloor \delta N^{1/2} \rfloor$, where $\delta>0$ would be decided later. We follow the proof of Chen and Fang \cite{C-F}.  For any $1\le j\le 4K$, let
		$$\mathcal{W}_j=\{w \in \mathcal{W}:w\equiv j\pmod{4K}\}.$$
		Then we have
		$$\mathcal{W}=\bigcup_{j=1}^{4K}\mathcal{W}_j$$
		with $\mathcal{W}_i\bigcap \mathcal{W}_j=\varnothing$ for $i\neq j$. Thus, for some positive integer $N_0$ depending on $\mathcal{W}$,
		\begin{align}\label{eq2-2}
			\sum_{n=1}^N\left(R_{\mathcal{S},\mathcal{W}}(n)-1\right) &=\sum_{n=1}^N\left(\sum_{j=1}^{4K}R_{\mathcal{S},\mathcal{W}_j}(n)-1\right)\nonumber\\
			&\ge  \sum_{n=1}^N\sum_{\substack{j=1\\ R_{\mathcal{S},\mathcal{W}_j}(n)\ge 1}}^{4K}\left(R_{\mathcal{S},\mathcal{W}_j}(n)-1\right) -N_0\nonumber\\&= \sum_{j=1}^{4K}\sum_{\substack{n=1\\ R_{\mathcal{S},\mathcal{W}_j}(n)\ge 1}}^{N}\left(R_{\mathcal{S},\mathcal{W}_j}(n)-1\right)-N_0.
		\end{align}
		By Lemma \ref{lem1} with $\mathcal{D}=\mathcal{W}_j$, we have
		\begin{align}\label{eq2-3}
			\sum_{\substack{n=1\\ R_{\mathcal{S},\mathcal{W}_j}(n)\ge 1}}^{N}\left(R_{\mathcal{S},\mathcal{W}_j}(n)-1\right)\ge \mathcal{W}_j\left(\delta_0N\right)-2
		\end{align}
		for any $1\le j\le 4K$, where $\delta$ and $\delta_0$ satisfy the restrictions
		\begin{align}\label{*}
			\begin{cases}
				\delta^2+\delta_0\le 1, &\\
				\frac{1}{16}\delta_0^2/\delta^2+\delta_0<1.
			\end{cases}
		\end{align}
		Combining (\ref{eq2-2}) and (\ref{eq2-3}), we obtain
		\begin{align}\label{eq2-4}
			\sum_{n=1}^N\left(R_{\mathcal{S},\mathcal{W}}(n)-1\right)\ge \sum_{j=1}^{4K}\left(\mathcal{W}_j\left(\delta_0 N \right)-2\right)-N_0=\mathcal{W}\left(\delta_0 N \right)-8K-N_0.
		\end{align}
		From (\ref{eq1-1}) we know that
		\begin{align}\label{eq2-5}
			\mathcal{W}\left(\delta_0 N \right)\ge \left(\frac{4}{\pi}\sqrt{\delta_0}-\varepsilon\right)N^{1/2}
		\end{align}
		for arbitrarily small $\varepsilon$ and thus we conclude that
		$$\sum_{n=1}^N\left(R_{\mathcal{S},\mathcal{W}}(n)-1\right)\ge \left(\frac{4}{\pi}\sqrt{\delta_0}-8\delta-\varepsilon\right)N^{1/2}-N_0$$
		by (\ref{eq2-4}) and (\ref{eq2-5}).
		A numerical calculation of $\frac{4}{\pi}\sqrt{\delta_0}-8\delta$ with the constrains (\ref{*}) by computer shows that $$\delta=0.022, \quad \delta_0=0.084$$
		and
		$$\frac{4}{\pi}\sqrt{\delta_0}-8\delta\ge0.19302$$
		are admissible.
	\end{proof}
	
	\section{Proof of Theorem \ref{theorem2}}
	\setcounter{lemma}{0}
	\setcounter{theorem}{0}
	\setcounter{corollary}{0}
	\setcounter{remark}{0}
	\setcounter{equation}{0}
	\setcounter{conjecture}{0}
	\begin{proof}[Proof of Theorem \ref{theorem2}]
		We follow the proof of \cite[Theorem 1]{Ding}. Suppose the contrary, i.e., 
		$$\limsup_{n\rightarrow \infty}\frac{\frac{\pi^2}{16}n^2-w_n}{n}=\beta<\frac{\pi}{4}+\frac{0.193\pi^2}{8}:=\gamma.$$
		Then there exists a number $n_1>0$ such that for $\sigma =\frac{1}{2}\left(\gamma-\beta\right)$ and all $n>n_1$,  we have
		$$\frac{\frac{\pi^2}{16}n^2-w_n}{n}\leq \beta+\frac{1}{2}\left(\gamma-\beta\right)=\gamma-\sigma,$$
		which means that
		$$w_n\geq\frac{\pi^2}{16}n^2-\left(\gamma-\sigma \right) n=\frac{\pi^2}{16}\left( n-\frac{8\gamma}{\pi^2}+\frac {8}{\pi^2} \sigma \right)^2-\frac{4}{\pi^2}\left(\gamma-\sigma\right)^2$$
		for all $n>n_1$. Thus there exists an integer $n_2\geq n_1$ such that for all $n>n_2$ we have
		\begin{align}\label{Formula}
			w_n\geq\frac{\pi^2}{16}\left( n-\frac{8\gamma}{\pi^2}+\frac {4}{\pi^2} \sigma \right)^2.
		\end{align}
		This would imply that
		\begin{align}\label{formula}
			\mathcal{W}(x)\leq\frac{4}{\pi}\sqrt{x}+\frac{8\gamma}{\pi^2}-\frac{4}{\pi^2}\sigma
		\end{align}
		for all $x>n_2$.  In fact, suppose that $\mathcal{W}(x)=\ell$, then from (\ref{Formula}) we get
		$$\frac{\pi^2}{16}\left( \ell-\frac{8\gamma}{\pi^2}+\frac {4}{\pi^2} \sigma \right)^2\le w_\ell\le x,$$
		from which (\ref{formula}) follows immediately.
		For any positive integer $N$, we have
		\begin{eqnarray}\label{3-1}
			\sum\limits_{n=1}^{N}R_{\mathcal{S},\mathcal{W}}(n)&=&\sum\limits_{\substack{m^2+w\leq N\\w\in \mathcal{W}}}1\nonumber\\&=&\sum\limits_{m\leq\sqrt{N}}\sum\limits_{\substack{w\leq N-m^2\\w\in \mathcal{W}}}1 \nonumber\\
			&=&\sum\limits_{m\leq\sqrt{N}}\mathcal{W}(N-m^2)\nonumber\\
			&\leq&\sum\limits_{m\leq\sqrt{N}}\left(\frac{4}{\pi}\sqrt{N-m^2}+\frac{8\gamma}{\pi^2}-\frac{4}{\pi^2}\sigma\right)+O(1)\nonumber\\
			&=&\frac{4}{\pi}\sum\limits_{m\leq\sqrt{N}}\sqrt{N-m^2}+\left(\frac{8\gamma}{\pi^2}-\frac{4}{\pi^2}\sigma\right)\sqrt{N}+O(1),
		\end{eqnarray}
		where the implied constant depends only on $n_2$.
		
		On the other hand, for square integers $N$, as in \cite[Section 2]{Ding} by the Euler--Maclaurian summation formula we have
		\begin{eqnarray}\label{3-2}
			\sum\limits_{m\leq\sqrt{N}}\sqrt{N-m^2}=\frac{\pi}{4}N-\frac{\sqrt{N}}{2}-
			\sum\limits_{k=0}^{\sqrt{N}-1}\int\limits_{k}^{k+1}\frac{t\left(\{t\}-\frac{1}{2}\right)}{\sqrt{N-t^2}}dt.
		\end{eqnarray}
		Furthermore, as in \cite[Section 2]{Ding} we still have
		\begin{eqnarray}\label{3-3}
			\int\limits_{k}^{k+1}\frac{t\left(\{t\}-\frac{1}{2}\right)}{\sqrt{N-t^2}}dt\ge0.
		\end{eqnarray}
		Combining  (\ref{3-2}) and (\ref{3-3}) gives us
		\begin{equation}\label{3-4}
			\sum\limits_{m\leq\sqrt{N}}\sqrt{N-m^2}\leq\frac{\pi}{4}N-\frac{\sqrt{N}}{2}.
		\end{equation}
		Hence by (\ref{3-1}) and (\ref{3-4}), we obtain that for large square integers $N$,
		\begin{equation}\label{3-5}
			\sum\limits_{n=1}^{N}R_{\mathcal{S},\mathcal{W}}(n)\leq
			N-\left(\frac{2}{\pi}-\frac{8\gamma}{\pi^2}+\frac{4}{\pi^2}\sigma\right)\sqrt{N}+O(1)\end{equation}
		which contradicts Theorem \ref{theorem1}.
	\end{proof}


\begin{thebibliography}{0}
		\bibitem{Abb} H.L. Abbott, {\it On the additive completion of sets of integers,} J. Number Theory, {\bf17} (1983), 135--143.
		\bibitem{Bal} R. Balasubramanian,{\it  On the additive completion of squares,} J. Number Theory, {\bf29} (1988), 10--12.
		\bibitem{B-R} R. Balasubramanian, D.S. Ramana, {\it Additive complements of the squares,} C. R. Math. Acad. Sci. Soc. R. Can., {\bf23} (2001), 6--11.
		\bibitem{B-S} R. Balasubramanian, K. Soundararajan, {\it On the additive completion of squares, II,} J. Number Theory, {\bf 40} (1992), 127--129.
		
		\bibitem{C-F} Y.-G. Chen, J.-H. Fang, {\it Additive complements of the squares,} J. Number Theory {\bf180} (2017), 410--422.
		\bibitem{Cil} J. Cilleruelo, {\it The additive completion of $k$--th powers,} J. Number Theory, {\bf44} (1993), 237--243.
		\bibitem{Ding} Y. Ding, {\it Green's problem on additive complements of the squares,} C. R. Math. Acad. Sci. Paris, {\bf358} (2020), 897--900.
		\bibitem{D-H} R. Donagi, M. Herzog, {\it On the additive completion of polynomial sets of integers,} J. Number Theory, {\bf3} (1971), 150--154.
		\bibitem{Erd} P. Erd\H os, {\it Problems and Results in Additive Number Theory Colloque sur la Th\'{e}orie des Nombres,} Bruxelles (1955), 127--137
		George Thone, Li\`{e}ge; Masson and Cie, Paris, 1956.
		\bibitem{Hab} L. Habsieger, {\it On the additive completion of polynomial sets,} J. Number Theory, {\bf51} (1995), 130--135.
		\bibitem{Mos} L. Moser, {\it On the additive completion of sets of integers,} Proceedings of Symposia in Pure Mathematics, vol. VIII, Amer. Math. Soc., Providence, R.I. (1965), 175--180.
		\bibitem{Ram} D.S. Ramana, {\it Some Topics in Analytic Number Theory,} PhD thesis University of Madras (May 2000).
		\bibitem{Ram2} D.S. Ramana, {\it A report on additive complements of the squares,} Number Theory and Discrete Mathematics, Trends Math., Chandigarh, 2000, Birkh\"{a}user, Basel (2002), 161--167.
		
		
	\end{thebibliography}
\end{document}